\documentclass[10pt]{article}
\usepackage{graphicx}
\usepackage{subfig}
\usepackage{lipsum}

\usepackage{amsmath,amsthm,amssymb}
\usepackage{mathrsfs}
\usepackage{mathtools}
\usepackage{thmtools} 
\usepackage{proof}    % Import proof.sty
\usepackage{enumitem}
\usepackage{algpseudocode,algorithm,algorithmicx}
\usepackage{booktabs} % makes prettier tables.
\usepackage{titlesec} % customize sections and subsections 
\usepackage{wrapfig}
\usepackage{parskip}
\usepackage{listings}% http://ctan.org/pkg/listings
\usepackage{dblfloatfix}
\usepackage{sidecap} % captions on the side
\usepackage[font=small, labelfont=bf, labelsep=period]{caption}  % caption formatting
\usepackage{cancel}

\usepackage{tikz}
\tikzset{every picture/.style={line width=0.5pt}} %set default line width to 0.75pt        
\font\bigbold=cmbx12

\def\maketitlefull#1#2#3#4#5{
  \centerline {\bigbold #1}
  \medskip
  \centerline {\bigbold #2}
  \medskip
  \centerline {\tensc #3}
  \medskip
  \centerline {\sl #4}
\centerline {\sl #5}
  \bigskip
}
\def\maketitle#1{
  \centerline {\bigbold #1}
}

\titleformat{\section}{\bf}{\thesection.}{0.5em}{}
\titlespacing{\section}{0em}{1.5\bigskipamount}{\medskipamount}
\titleformat{\subsection}{\bf}{\thesubsection.}{0.5em}{}
\titlespacing{\subsection}{0em}{\bigskipamount}{\medskipamount}
\titleformat{\paragraph}[runin]{\bf}{}{0em}{}[.]
\titlespacing{\paragraph}{0em}{\medskipamount}{*0.75}

\font\tensc=cmcsc10

\def\case#1. {{\it Case #1}.\enspace}

\def\parlabel#1. {\medskip\noindent{\bf #1.}\enspace}

\def\sqr#1#2{{\vcenter{\vbox{\hrule height.#2pt
        \hbox{\vrule width.#2pt height#1pt \kern#1pt
          \vrule width.#2pt}
        \hrule height.#2pt}}}}

\def\slug{\quad\hbox{\kern1.5pt\vrule width2.5pt height6pt depth1.5pt\kern1.5pt}\medskip}

\setlist[enumerate]{itemsep=\smallskipamount,parsep=0pt,label={\rm \roman*)}}
\setlist[itemize]{itemsep=\smallskipamount,parsep=0pt}
\newtheorem{thm}{Theorem}

\newtheorem{prop}[thm]{Proposition}
\newtheorem{lem}[thm]{Lemma}
\newtheorem*{lem*}{Lemma}

\newtheorem{conj*}{Conjecture}

\theoremstyle{definition}

\theoremstyle{remark}

\newcommand{\E}{\mathbb{E}}

\usepackage{bbm}

\DeclarePairedDelimiterX{\inp}[2]{\langle}{\rangle}{#1, #2} %inner product
\newcommand{\no}{\noindent}

\usepackage{calligra}
    \DeclareMathAlphabet{\mathcalligra}{T1}{calligra}{m}{n}
    \DeclareFontShape{T1}{calligra}{m}{n}{<->s*[2.2]callig15}{}

\renewcommand{\mathbb}{\mathbf}

\begin{document}

\maketitlefull{A note on the exact simulation of a random eigenvalue}{of a GUE matrix}{Luc Devroye$^*$ and Jad Hamdan$^\dag$}{$^*$School of Computer Science, McGill University,\,\,}{$^\dag$Mathematical Institute, University of Oxford}
\medskip

% Proposal abstract
\[
  \vbox{
    \hsize 5.5 true in
    \noindent{\bf Abstract.}\enskip 
        We develop a simple algorithm to generate random variables described by densities equaling squared Hermite functions. As an application, we show how to generate a randomly chosen eigenvalue of a matrix from the Gaussian Unitary Ensemble ({\textsc{gue}}) in sub-linear expected time.
    \smallskip 
    
    
    % \noindent{\bf MSC2020 subject classification.} Primary 62G07, 68Q87, 68W40; secondary 60C05, 60J85.
  }
\] 

\smallskip

\section{Introduction} \label{sec:intro}
\no 

\textsc{In this note,} we concern ourselves with the generation of a random eigenvalue of a matrix in the Gaussian Unitary Ensemble {\textsc{gue}}$(n)$, described by the Gaussian measure with density
\[
    (2^{n/2}\pi^{n^2/2})^{-1}e^{-(n/2) \, \text{tr} H^2}
\]
on the space of $n\times n$ Hermitian matrices $H=(H_{ij})_{i,j=1}^n$. That is, the diagonal elements are normal with zero mean and unit variance, and the off-diagonal elements are conjugate complex with independent real and imaginary parts that are both zero mean normals with variance $1/2$. For more on {\textsc{gue}} matrices, see \cite{guionnet, deift}.

We operate under the {\sc{ram}} model: real numbers can be stored and operated upon in constant time. We also assume that all the standard operations as well as the exponential, logarithmic, trigonometric, and gamma functions can be computed in constant time. Finally, a source capable of producing an i.i.d.\ sequence of uniform $[0, 1]$ random variables is available.

The problem of computing eigenvalues of a (deterministic) matrix is classical. This yields a natural first group of methods for the random case: one starts from the matrix $H$ and uses iterative numerical methods to compute the set of eigenvalues, noting that 
the Abel-Ruffini theorem (see \cite{abel}, \cite{ruffini}) implies that there is no exact finite-time algorithm that finds roots of polynomials of degree greater than 5.
% The first group of methods starts from the matrix $H$ and uses iterative numerical methods to compute the set of eigenvalues. 
These methods must take time $\Omega (n^2)$.
If one is satisfied with a given accuracy level, then some of these approximations can be computed in time $O(n^2)$.
Prominent among these numerical methods are the QR algorithm \cite{qr, qr2, qr3, watkins} and its variations, Lanczos' algorithm \cite{lanczos}, the Rayleigh quotient iteration \cite{kress} for Hermitian matrices, or, more generally, power iteration \cite{powerit}. 

The second group of methods uses numerical approximation but starts from a different random matrix $H'$ with the property that the vector of eigenvalues has the same distribution as that of $H$. Following Dumitriu and Edelman \cite{dumitriu+edelman}, we note that one can take $H'$ tridiagonal, with i.i.d.\ normal $(0,1)$ random variables $N_i$ on the diagonal, and with two identical adjacent diagonals filled with independent random variables  $\sqrt{G_1},\ldots, \sqrt{G_{n-1}}$, where $G_i$ is gamma $(i)$. In other words, $H'_{i,i} = N_i$, and $H'_{i,i+1}=H'_{i+1,i}=\sqrt{G_i}$, $1 \le i \le n-1$, and so $H'$ is a random Jacobi matrix.  Generating $H'$ takes linear time. In addition, a divide-and-conquer algorithm \cite{coakley} can compute the spectrum of $H'$ in time $O(n \log n)$ if one only requires the answer up to a fixed precision.

The third group of methods uses stochastic tools to approximate the vector of eigenvalues, based on Markov chain convergence, with rates of convergence dependent upon the mixing times of these Markov chains or Gibbs samplers \cite{krishna, gautier2021}. 

The fourth gaggle of methods relies on the observation that for some integrable random matrix ensembles (of which the \textsc{gue} is an example), the eigenvalues form a so-called determinantal point process (\textsc{dpp}) \cite{Macchi}. This is useful, as the literature on sampling \textsc{dpp}s is vast and includes both approximate \cite{KuleszaTaskar, affandi, LiSraJegelka} and \textit{exact} algorithms \cite{peres, launay2020exact}. 
The main exact algorithm, sometimes referred to as \textsc{hkpv} \cite{peres}, draws the entire vector of eigenvalues of $H$ by sampling from a sequence $\{p_i(x)\}_{i=1}^n$ of appropriately constructed marginals. This procedure has complexity $O(n^3)$ assuming one can sample from each $p_i$ in constant time, but recent surveys \cite{lavancier-, gautier2021, lille} report that these sampling subroutines have a poorly understood complexity. Noting that $p_1$ is precisely the density of a randomly selected eigenvalue of $H$, we still have to find an efficient and exact sampler for this quantity, which we recall is our primary objective.

For $n\leq 4$, one can simply generate $H$ and compute the roots of its characteristic polynomial. Thus, we will describe a method for $n\geq 5$. 
In Section 2, we develop an algorithm to sample a random uniformly selected \textsc{gue}$(n)$ eigenvalue and prove that its expected run-time is sublinear in $n$. We rely on the fact that our problem can be reduced to generating random variables from densities which can be expressed in orthogonal polynomials. We then sample these using the rejection method. Earlier investigations in this direction \cite{baskerville2022, baskerville2022b} used a different proposal density and were strictly empirical, leaving the question of explicit rejection bounds open before this work.  To our knowledge, no other exact algorithm with sublinear expected time is available today. 

\medskip
\subsection{The density of a random GUE eigenvalue}

For a distribution on sets of $n$ points, the $k-$point correlation function $\sigma_k(x_1,...,x_k)$ describes the induced probability distribution on uniformly selected subsets of size $k\leq n$. It is typically normalized so that $\frac{1}{n}\sigma_n(x)$ is the probability density function of a uniformly selected (set of $1$) point. It is well-known  (see \cite{edelman}) that for {\textsc{gue}}$(n)$ eigenvalues, the $k-$point correlation function is given by a $k\times k$ determinant $\det(K(x_1,x_2))_{1\leq i,j\leq k}$, where
\[
    K(x,y) = \frac{1}{\sqrt{2\pi}}\sum_{k=0}^{n-1}\frac{H_k(x)H_k(y)e^{-(x+y)^2/4}}{k!} = \sum_{k=0}^{n-1}\phi_k(x)\phi_k(y).
\]
In the expression above, $H_k$ is the $k$-th \textit{Hermite polynomial}
\[
    H_k(x) = (-1)^ke^{x^2/2}\bigg(\frac{\mathrm{d}}{\mathrm{d}x}\bigg)^ke^{-x^2/2},
\]
and $\phi_k$ is the so-called \textit{Hermite function} $H_k(x)e^{-x^2/4}/\sqrt{k!\sqrt{2\pi}}$ (in the theory of orthogonal functions, $K(x,y)$ is known as the \textit{Christoffel-Darboux} kernel \cite{szego}). The probability density for the distribution of a single, uniformly selected eigenvalue from a {\textsc{gue}}$(n)$ matrix is thus equal to 
\begin{align}\label{density}
   \frac{1}{n}K(x,x)=\frac{1}{n}\sum_{k=0}^{n-1}\phi_k(x)^2.
\end{align}
We remark in passing that, using asymptotic properties of Hermite polynomials (see \cite{mehta}), one recovers Wigner's famous semicircle law \cite{wigner} from this expression in the large--$n$ limit.

Next, we note that the Hermite polynomials are orthogonal. In particular, for any $k\in \mathbb{N}$, we have 
\[
    \int_{-\infty}^\infty H_k(x)^2e^{-x^2/2}\mathrm{d}x = \sqrt{2\pi}\, k!
\]
(see \cite{sansone} for a proof), which shows that $\phi_k^2$ is a density for every $k$. It then follows from (\ref{density}) that if one picks an index $k\in \{0,\dots,n-1\}$ uniformly at random and generates a random variate described by the density $\phi_k^2$, the result would be distributed as a uniformly selected {\textsc{gue}}$(n)$ eigenvalue. The remainder of this paper focuses on generating random variables described by the densities $\phi_k^2$. 

\section{Generating random variates with squared Hermite densities}

\subsection{Notation and preliminaries on Hermite polynomials}

The aforementioned Hermite polynomials can alternatively be defined using the recurrence $H_0=1$, $H_1(x)=x$ and 
\begin{equation}\label{recurrence}
    H_{k+1}(x)=xH_k(x)-kH_{k-1}(x)
\end{equation}
for any $k\geq 2$, giving us a simple $O(k)$ algorithm to compute $H_k$ for fixed $x$ (\cite{szego}).
We also have the following theorem, due to Bonan \& Clarke \cite{bonanclarke}, which gives an upper bound for $\phi_n^2$ on $\mathbb{R}$. Note that we computed the explicit constants in that paper. Other inequalities could also have been used, such as those developed by Foster and Krasikov \cite{foster} and Krasikov \cite{krasikov}.

\begin{thm}\label{upperbound}
    For any $n\in \mathbb{N}$, let $\phi_n^2$ be the square of the $n-$th Hermite function (defined in the previous section). Then
    \begin{align*}
        \sup_{x\in \mathbb{R}} \phi_n^2(x) \leq \frac{8(\pi+1)}{3n^{1/6}} .
    \end{align*}
    Furthermore,
    \begin{align*}
        \phi_n^2(x) \leq \begin{cases}
         \frac{8\pi}{3\sqrt{4n+2-x^2}}, & \text{if $|x|\leq \sqrt{4n+2}$,} \\
        \frac{2\sqrt{2}B^2} {n^{5/6}\,(\sqrt{4n+2}-x)^{4}}, & \text{if $|x|> \sqrt{4n+2}$},
    \end{cases}
    \end{align*}
    where $B=(\pi+1)^2\sqrt{8(\pi+1)/3}$.
\end{thm}

If we let 
\[
    x_1=x_1(n)=\sqrt{4n+2-\frac{\pi^2}{(\pi+1)^2}\,n^{1/3}}
\]
and
\[
    x_2=x_2(n)=\sqrt{4n+2}+\sqrt{B}\bigg(\frac{3}{2\sqrt{2}(\pi+1)}\bigg)^{1/4}n^{-1/6},
\]
it follows directly from the theorem above that the function 
\begin{align*}
    h_n(x) \stackrel{\text{def}}{=} \begin{cases}
       \frac{8\pi}{3\sqrt{4n+2-x^2}}, & \text{if $|x|\leq x_1$,} \\
        \frac{8(\pi+1)}{3 n^{1/6}} ,       & \text{if $|x|\in (x_1, x_2]$,} \\
        \frac{2\sqrt{2}B^2} {n^{5/6}\,(\sqrt{4n+2}-x)^{4}},   & \text{if $|x|> x_2$,}
    \end{cases}
\end{align*}
dominates $\phi_n^2$ on all of $\mathbb{R}$. 

The following lemma will be useful for analyzing our algorithm's expected runtime.

\begin{lem}\label{bonanintegral} Let $h_n$ be defined as above. Then
    \[
        \int_\mathbb{R} h_n(x) \,\mathrm{d}x = O(1), \quad \int_{x_1}^{\infty} h_n(x) \, \mathrm{d}x = O(n^{-1/3}).
    \]
\end{lem}
\begin{proof}
    Let $n\in \mathbb{N}$ be arbitrary, then
    \begin{align*}
        &\int_\mathbb{R}h_n \\
        &= \int_{0}^{x_1} \frac{8\pi}{3}\frac{1}{\sqrt{4n+2-x^2}}\mathrm{d}x + \int_{x_1}^{x_2} \frac{8(\pi+1)}{3n^{1/6}} \,\mathrm{d}x + \int_{x_2}^\infty \frac{2\sqrt{2}B^2n^{-5/6}}{(\sqrt{4n+2}-x)^4}\mathrm{d}x \\
        &= \frac{8\pi}{3}\arcsin{\frac{x_1}{\sqrt{4n+2}}} + \frac{8(\pi+1)}{3n^{1/6}}|x_2-x_1|+ \frac{\sqrt{B}\,2\sqrt{2}}{3n^{1/3}}\bigg(\frac{2\sqrt{2}(\pi+1)}{3}\bigg)^{3/4}
    \end{align*}
    and the lemma follows from the fact that $|x_2-x_1|=O(n^{-1/6})$.
\end{proof}
We will also need the following representation of Hermite polynomials due to van Veen \cite{vanveen}.
\begin{thm} \label{vv}
    For any $n, k\in \mathbb{N}$, $x\in [0,2\sqrt{n+1}]$, 
    \begin{align*}
        H_n(x)&= A_n(x)\big(B_n(x)+\mu R_n(x)\big)
    \end{align*}
    where 
    \begin{align*}
        &A_n(x) \overset{\mathrm{def}}{=} \frac{\Gamma(n+1)}{\pi}\frac{e^{{(n+1)}/2+x^2/4}}{(n+1)^{n/2}}, \\
        &B_n(x) \stackrel{\mathrm{def}}{=}  \frac{\sqrt{\pi}}{\left\{(n+1)\sin(\alpha)\right\}^{1/2}}*\sin\left\{\frac{(n+1)}{2}\big(\sin(2\alpha)-2\alpha\big)+\frac{\alpha}{2}+\frac{3\pi}{4}\Big)\right\},\\
        & R_n(x) \overset{\mathrm{def}}{=} 
        \frac{1}{3(n+1)\sin(\alpha)^2},\\
        &\mu\overset{\mathrm{def}}{=} 3.1+1.1*(\sin(\alpha/2+3*\pi/4)/\sin(\alpha))^2,
    \end{align*}
 and
\[
\alpha = \alpha(x) = \arccos{\frac{x}{2\sqrt{n+1}}}.
\]

\end{thm}

\begin{figure}
    \centering
    \includegraphics[width=0.6\textwidth]{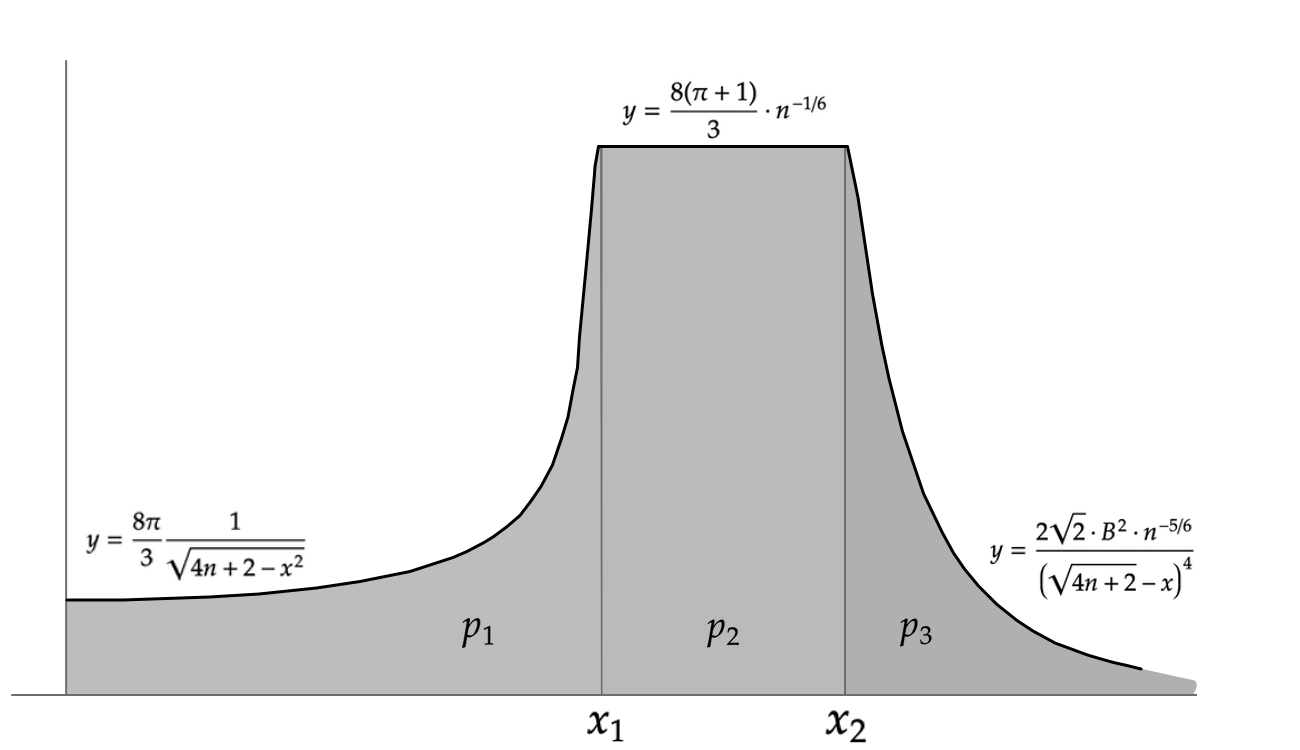}
    \caption{The function $h_n$, defined piecewise. As in Algorithm 1 (section \ref{bonangen}), $p_1$, $p_2$ and $p_3$ denote the area under the curve in $[0,x_1)$, $[x_1,x_2)$ and $[x_2,\infty)$ respectively. }
\end{figure}

This representation is then extended to $|x|\leq 2\sqrt{n+1}$ using the fact that $H_n$ is even. On said domain, for any $n, k\in \mathbb{N}$, we use $V_n\stackrel{\text{def}}{=}A_n*B_n$ (as defined in the theorem above) to approximate $H_n$ and define the following approximation to $\phi_n^2$:
\[
    f_n(x) \stackrel{\text{def}}{=} \frac{V_n(x)^2e^{-x^2/2}}{{\sqrt{2\pi}\,n!}} 
\]
(we define $f_n=0$ if $|x|>2\sqrt{n+1}$). Note that for any $|x|<2\sqrt{n+1}$,
\begin{align*}
    |\phi_n^2(x)-f_n(x)| = \bigg(\big(\mu R_n(x)\big)^2+2\mu R_n(x)B_n(x)\bigg)\frac{A_n^2(x)e^{-x^2/2}}{\sqrt{2\pi} \, n!}.
\end{align*}
\begin{figure}
    \centering
    \includegraphics[width=0.6\textwidth]{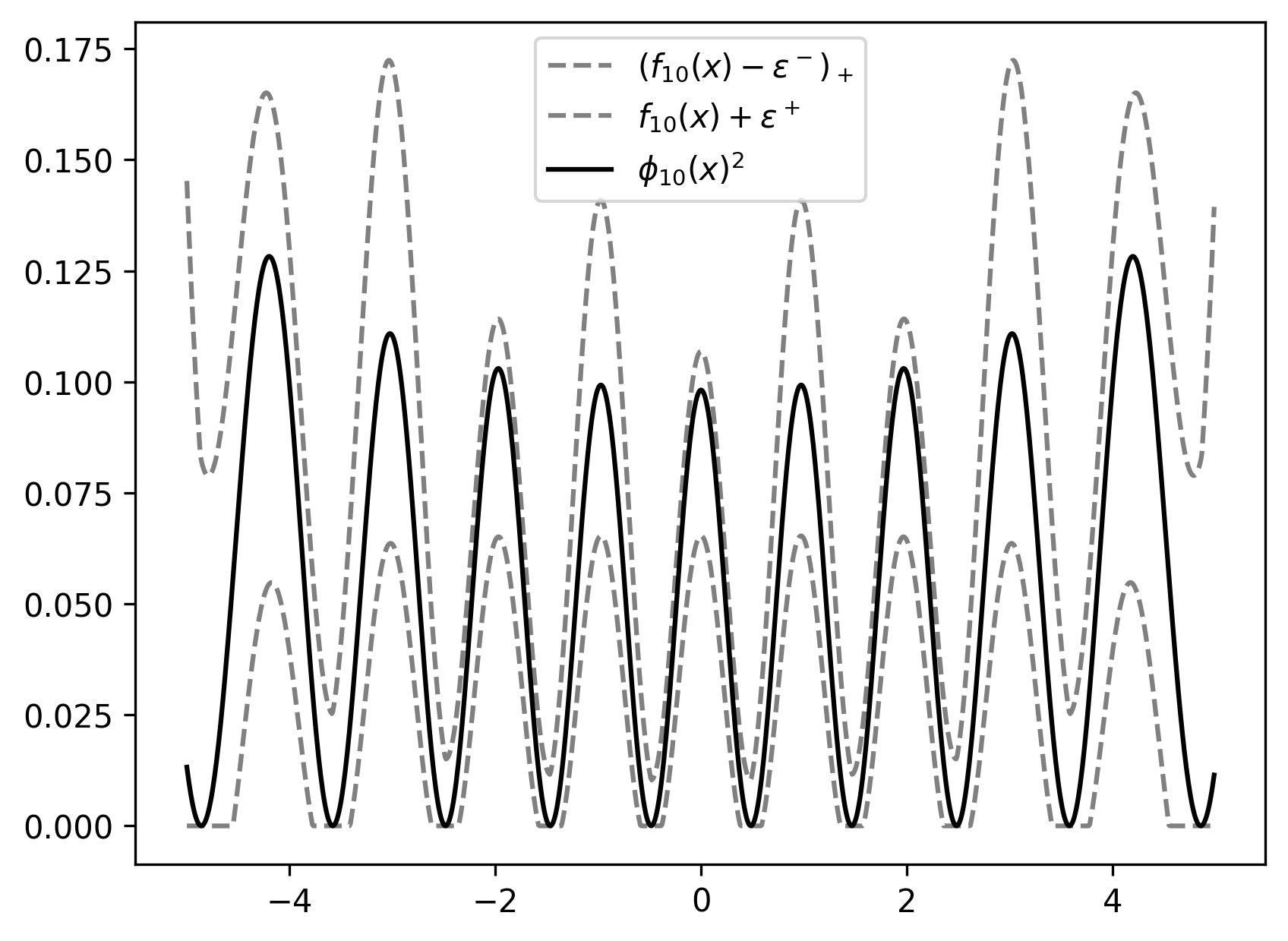}
    \caption{The function $\phi_{10}^2$ (bold) bounded from below and above by $(f_{10}-\epsilon^{-})_{+}$ and $(f_{10}+\epsilon^{+})$  (dashed).}
    \label{fig:my_label}
\end{figure}
Using the fact that $|\mu|<4.2$, it follows that
\begin{align*}
    (\phi_n^2(x)-f_n(x))_+ \,\leq \,&\epsilon^{+}(x)\stackrel{\text{def}}{=}\frac{A_n^2(x)e^{-x^2/2}}{\sqrt{2\pi} \, n!} \big((8.4)(B_n(x))_+R_n(x)+(4.2)^2R_n(x)^2\big) \\
    (\phi_n^2(x)-f_n(x))_- \,\leq \,&\epsilon^{-}(x)\stackrel{\text{def}}{=} \frac{A_n^2(x)e^{-x^2/2}}{\sqrt{2\pi} \, n!} (8.4)\bigg(\Big(B_n(x)R_n(x)\Big)_+\lor\Big(B_n(x)R_n(x)\Big)_{-}\bigg) 
\end{align*}
(we used $a\wedge b$ (resp. $a\lor b$) to denote $\min(a,b)$ (resp. $\max(a,b)$) and $(x)_{+}=(x\lor 0)$, $(x)_{-}=(-x\lor 0)$).
Note that $\epsilon^{+/-}$ are positive by definition and
\begin{equation}\label{upperlower}
        \left(f_n-\epsilon^{-}\right)_{+} \leq \phi_n^2 \leq \left(f_n +\epsilon^{+}\right) \wedge h_n
\end{equation}
for any $x\in\mathbb{R}$. Lastly, we define $\Delta_\epsilon(x)$ to be the gap between the upper and lower bound for $\phi_n^2(x)$ and note that there exists a constant $C$ for which
\begin{align*}
\Delta_\epsilon(x) &\leq C\bigg(\big|R_n(x)|^2+ |R_n(x)B_n(x)|\bigg)\frac{A_n^2(x)e^{-x^2/2}}{\sqrt{2\pi} \, n!}
\end{align*}
and the following proposition holds.
\begin{prop}\label{vverror}
    For large enough $n$, 
    \[
        \int_{0}^{x_1}\Delta_\epsilon(x) \,\mathrm{d}x=O\big(n^{-1/3}\big).
    \]
\end{prop}
\begin{proof}
    See the appendix.
\end{proof}

\subsection{Generating random variates with density $h_n/\int h_n$} \label{bonangen}

% TO DO : FIX THE INVERSE SINE SCALING

Our main algorithm in the following section requires the generation of random variates with density $h_n/\int h_n$. This can be done in constant expected time using the following algorithm, which is a straightforward application of the inversion method (see \cite{rng} for an exposition of this method). 

\begin{lstlisting}
Algorithm 1:
    Set $B \leftarrow (\pi+1)^2\sqrt{8(\pi+1)/3}$.
    Set $p_1 \leftarrow \int_0^{x_1}h_n=\frac{8\pi}{3}\arcsin{\frac{x_1}{\sqrt{4n+2}}}$.
    Set $p_2 \leftarrow \int_{x_1}^{x_2}h_n=\frac{8(\pi+1)}{3}n^{-1/6}\big|x_2-x_1\big|$.
    Set $p_3 \leftarrow \int_{x_2}^\infty h_n=\sqrt{B}\cdot\left(\frac{2\sqrt{2}}{3}\right)^{7/4}(\pi+1)^{3/4}n^{-1/3}$.
    Generate $S$, where $\P\{S=1\}=\P\{S=-1\}=1/2$.
    Generate $U$ uniformly on $[0,1]$.
    If      $U< p_1/(p_1+p_2+p_3)$ then 
            Generate $V$ uniformly on $[0,1]$.
            $X \leftarrow \sqrt{4n+2} \,\,\big|\sin(V*\arcsin(x_1/\sqrt{4n+1}))\big|$
    else if $p_1/(p_1+p_2+p_3)\leq U<(p_1+p_2)/(p_1+p_2+p_3)$ then
            Generate $V$ uniformly on $[0,1]$.
            $X\leftarrow x_1+(x_2-x_1)*V$
    else 
            Generate $V$ uniformly on $[0,1]$.
            $X\leftarrow \sqrt{4n+2}\,+\,{\sqrt{B}\bigg(\frac{3}{2\sqrt{2}(\pi+1)}\bigg)^{1/4}n^{-1/6}}*{V^{-1/3}}$
    return $X*S$
\end{lstlisting}

% Moving on to runtime analysis, it suffices to show that if the first condition is satisfied, one has $O(1)$ iterations. The number of iterations is distributed like a geometric random variable with parameter 
% \[
%     p=\arcsin{\sqrt{1-\bigg(\frac{\pi}{\pi+1}\bigg)^2\frac{n^{1/3}}{(4n+2)}}},
% \]
% and we can therefore expect $1/p$ iterations. The derivative of 
% \[
%     \arcsin{\sqrt{1-\bigg(\frac{\pi}{\pi+1}\bigg)^2\frac{x^{1/3}}{(4x+2)}}}
% \]
% is negative on $(0,1/4)$ and positive on $[1/4,\infty)$, from which it follows that $p$ is smallest for $n=1$ giving $1/p\leq 1/\arcsin(\sqrt{1-\pi^2/(6(\pi+1)^2)}) \leq 1$.

\subsection{A first eigenvalue algorithm that is linear in $n$}\label{bonanalg}

For any $n$, we now know how to generate from $h_n/\int h_n$ and that $h_n$ dominates $\phi_n^2$. The following rejection algorithm (see \cite{rng} and \cite{vonneumann}) will be used as a stepping stone towards our main algorithm, and will be shown to run in linear time in the following section.
\begin{lstlisting}
Algorithm 2:
    Repeat until Accept
        Generate $X$ from $h_n/\int h_n$
        Generate $U$ uniformly on $[0,1]$
        Accept $\leftarrow \big[U*h_n(X)\leq \phi_n^2(X)\big]$         
    return $X$
\end{lstlisting}
Note that the comparison requires us to compute $\phi_n^2(X)$, which we do using the recurrence relation for Hermite polynomials (\ref{recurrence}) given earlier.
\subsection{Main algorithm and runtime analysis}\label{bonanvvalg}

We can refine this second algorithm using van Veen's estimate for $H_n$ stated in Theorem \ref{vv}. Let $\phi_n^2$, $f_n$, $h_n$, and $\epsilon^{+/-}$ be defined as above.

% \clearpage
\begin{lstlisting}
Algorithm 3:
    Repeat forever
        Generate $X$ from $h_n/\int h_n$
        Generate $U$ uniformly on $[0,1]$
        If      $U*h_n(X) \leq \left(f_n(X)-\epsilon^{-}(X)\right)_{+}$ then return X
        else if $U*h_n(X) \leq \left(f_n(X)+\epsilon^{+}(X)\right)$ then
                if $U*h_n(X)\leq \phi^2_n(X)$             then return X
\end{lstlisting}
We emphasize that $f_n, h_n$ and $\epsilon^{+/-}$ can be computed in $O(1)$ time, and note that in the last line, $\phi_n^2$ is computed using the recurrence formula given above, as in algorithm 2.

Fix $n\in \mathbb{N}$, and let $N$ be the number of iterations of algorithm 3 when generating from $\phi_n^2$. Let $T_i$ be the time taken by the $i$-th iteration, where $1\leq i\leq N$, then the total runtime of the algorithm is $\sum_{i=1}^N T_i$. Since $N$ is a stopping time, applying Wald's identity (see \cite{williams}) yields
\begin{align*}
    \E\bigg\{\sum_{i=1}^NT_i\bigg\}=\E\{N\}\E\{T_1\}.
\end{align*}
We have
\[
    \E\{N\} = \int h_n = O(1),
\]
and using the fact that $\phi_n^2$ can be computed in $O(n)$ time with the recurrence for $H_n$ explained above, and that one can sample from $h_n/\int h_n$ in constant time using algorithm 2 (cf. \ref{bonangen}), we conclude that
\begin{align*}
    \E\{T_1\}= O\Bigg(n\cdot\bigg(\int_{0}^{x_1}\Delta_\epsilon(x)\,\mathrm{d}x+\int_{x_1}^{\infty}h_n(x)\,\mathrm{d}x\bigg)\Bigg).
\end{align*}

  It then follows from lemmas \ref{vverror} and \ref{bonanintegral} that $\E\{T_1\}=O(n^{2/3})$, and, in turn, that this refined algorithm is sublinear with an expected runtime of the same order.
On the other hand, replacing $f_n$ with $0$ and $\epsilon^+$ with $\phi_n^2$ reduces the above algorithm to algorithm 2 and shows that the latter runs in linear expected time.

\section{Appendix: Proof of Proposition \ref{vverror}.}
We denote $f(x)=O(g(x))$ by $f\ll g$. 
    It suffices to show that
  \[
        I(0,x_1)=\int_0^{x_1}\bigg( |R_n(x)B_n(x)|+R_n(x)^2\bigg)\frac{A_n^2(x)e^{-x^2/2}}{\sqrt{2\pi} \, n!}\, dx
  \]
  is of the desired order, since $\Delta_\epsilon \ll I(0,x_1)$.
We begin with an application of Stirling's approximation, which yields
\[
    \frac{A_n(x)^2e^{-x^2/2}}{\sqrt{2\pi}n!} \ll {\sqrt{n}}.
\]
If $|x|\leq \sqrt{3(n+1)}$, we have $\sin\alpha>1/2$, $B_n(x)=O(n^{-1/2})$ and $R_n(x)=O(n^{-1})$, so that
\[
    I\bigg(0,\sqrt{3(n+1)}\bigg)\ll \sqrt{n}\int_{0}^{\sqrt{3(n+1)}} \big(|B_n(x)R_n(x)|+R_n(x)^2 \big)\, \mathrm{d}x = O\left(\frac{1}{\sqrt{n}}\right).\]
To bound $I(\sqrt{3(n+1)},x_1)$, use Stirling's approximation as above,
\[
    I(\sqrt{3(n+1)}, x_1) \ll \int_{\sqrt{3(n+1)}}^{x_1} \frac{1}{n\sin(\alpha(x))^{5/2}}
    +\frac{1}{n^{3/2}\sin(\alpha(x))^{4}}\, \mathrm{d}x, 
\]
apply the following technical lemma (taking $\xi=1/3$) and conclude that the integral is $O(1/n^{1/3})$.

\begin{lem} \label{sinlemma}
   Let $x=\sqrt{4(n+1)-\gamma n^\xi}$ for some $\xi \in (0,1)$ and any $\gamma\in\mathbb{R}_{>0}$ for which $\gamma n^\xi <4(n+1)$. Then for $\beta>2$,
    \[
        \int_{0}^x \sin(\alpha(t))^{-\beta}\,\mathrm{d}t \ll n^{(\xi-1/2)+\beta(1-\xi)/2}.
    \]
   for $n$ large enough. 
\end{lem}

\medskip
\noindent{\textsc{proof of proposition \ref{vverror}.}
    Recall that $\alpha(t)=\arccos{\big(t/\sqrt{4(n+1)}\big)}$ for $0 \leq t <\sqrt{4(n+1)}$, giving
    \[
        -\sin\alpha \,\mathrm{d}\alpha = \frac{\mathrm{d}t}{\sqrt{4(n+1)}}
    \]
    and in turn, for any $x=\sqrt{4(n+1)-\gamma n^\xi}$ where $\gamma\in\mathbb{R}_{>0}$, $\xi \in (0,1)$,
    \begin{align*}
        \int_{0}^{x}\sin \big(\alpha(t)\big)^{-\beta} \mathrm{d}t &= -\int_{\arccos{\big(x/\sqrt{4(n+1)}\big)}}^{\pi/2}\sqrt{4(n+1)}(\sin\alpha)^{1-\beta}\mathrm{d}\alpha.\\
    \end{align*}
    Since $\alpha\in (0,\pi/2]$, we have 
    \[
        \frac{2}{\pi} \leq \frac{\sin(\alpha)}{\alpha} \leq 1,
    \]
    from which it follows that the integral above is smaller than
    \begin{equation} \label{intqt}
        C\sqrt{4(n+1)}*\left(\frac{1}{\arccos{\big(x/\sqrt{4(n+1)}\big)}}\right)^{\beta-2}
    \end{equation}
    for a constant $C$ that can depend on $\beta$, assuming $\beta>2$. As $\arccos(x)\geq \sqrt{1-x^2}$ for any $|x|\leq 1$, if $x=\sqrt{4(n+1)(1-\delta)}$ for $\delta<1$, (\ref{intqt}) is bounded from above by
    \begin{align*}
        C\sqrt{4(n+1)}*\left(\frac{1}{\sqrt{1-x^2/{4(n+1)}}}\right)^{\beta-2} \leq C\sqrt{4(n+1)}\,\delta^{-(\beta-2)/2} \ll \,\sqrt{n\delta^{-(\beta-2)}}.
    \end{align*}
    The lemma then follows after setting $\delta=n^{\xi-1}$ for any $\xi\in(0,1)$.

\section{Appendix: Simulations} The following histograms were obtained by sampling a random \textsc{gue}(n) eigenvalue for different $n$ using our algorithm. In each case, 100000 points were sampled. The line in black corresponds to the density
\[
    \frac{1}{n}\sum_{k=0}^{n-1}\phi_k(x)^2.
\]
\clearpage
\begin{figure*}[t!]
        \subfloat[$n=5$]{%
            \includegraphics[width=.48\linewidth]{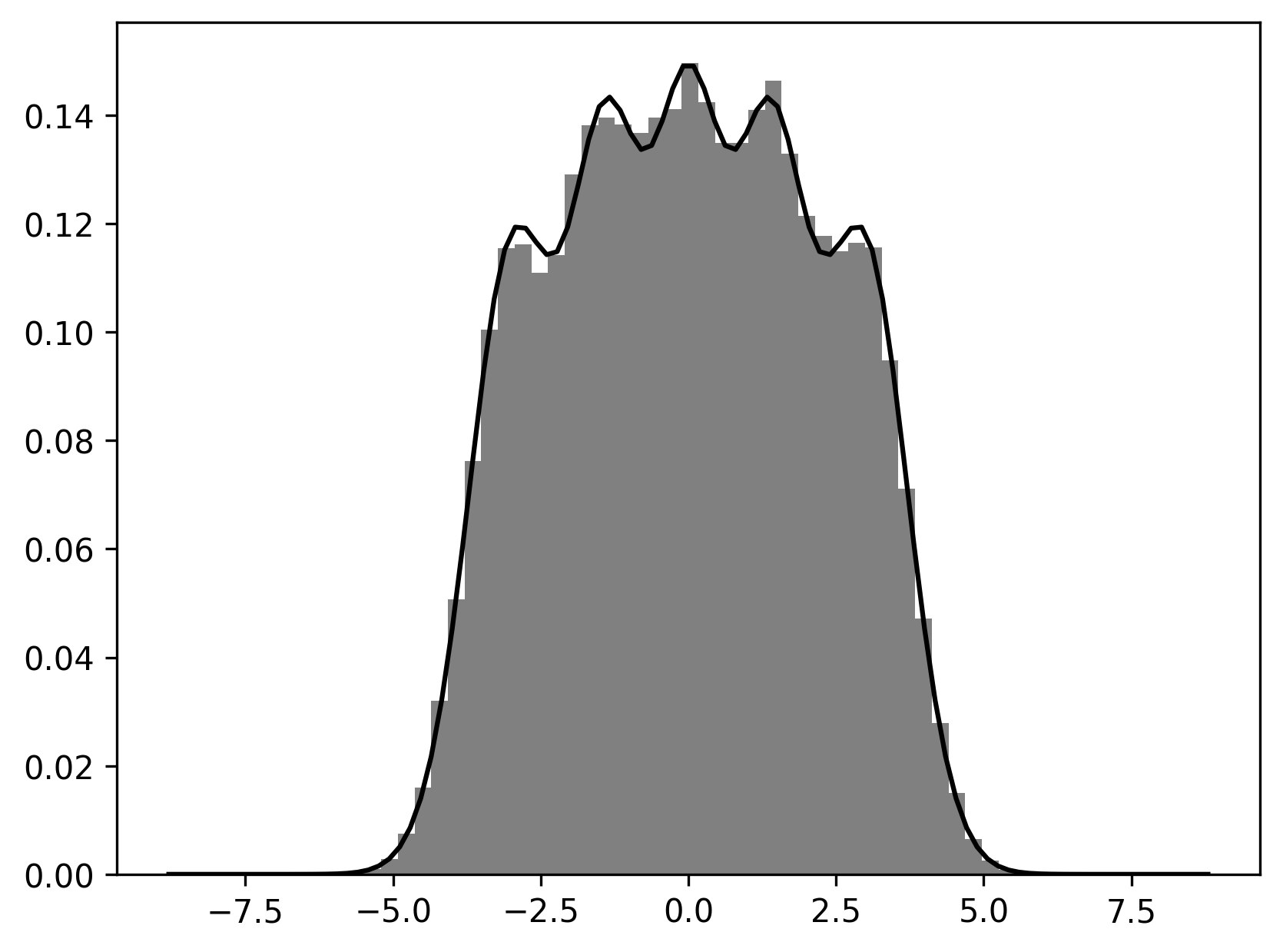}%
            \label{subfig:a}%
        }\hfill
        \subfloat[$n=10$]{%
            \includegraphics[width=.48\linewidth]{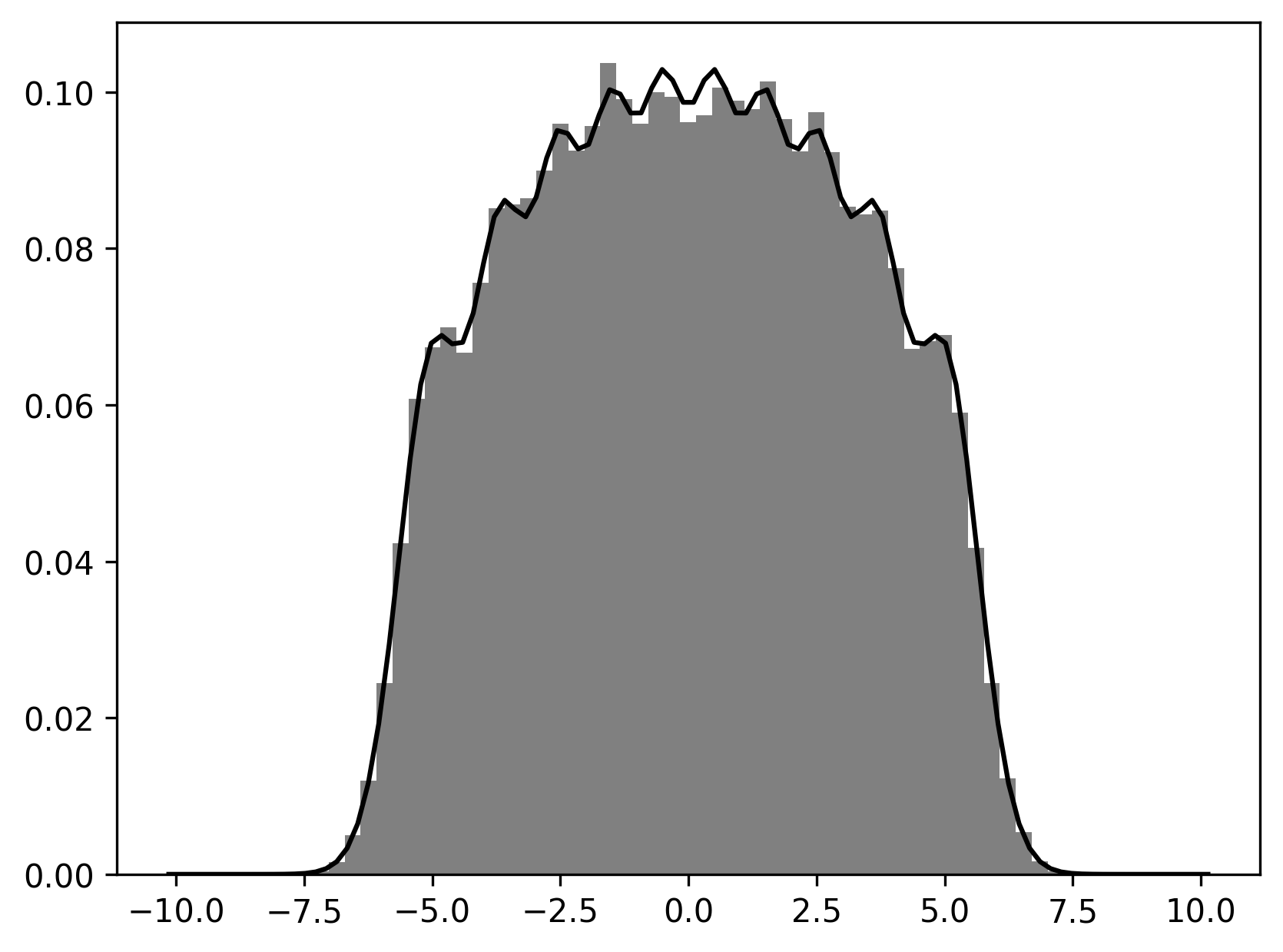}%
            \label{subfig:b}%
        }\\
        \subfloat[$n=15$]{%
            \includegraphics[width=.48\linewidth]{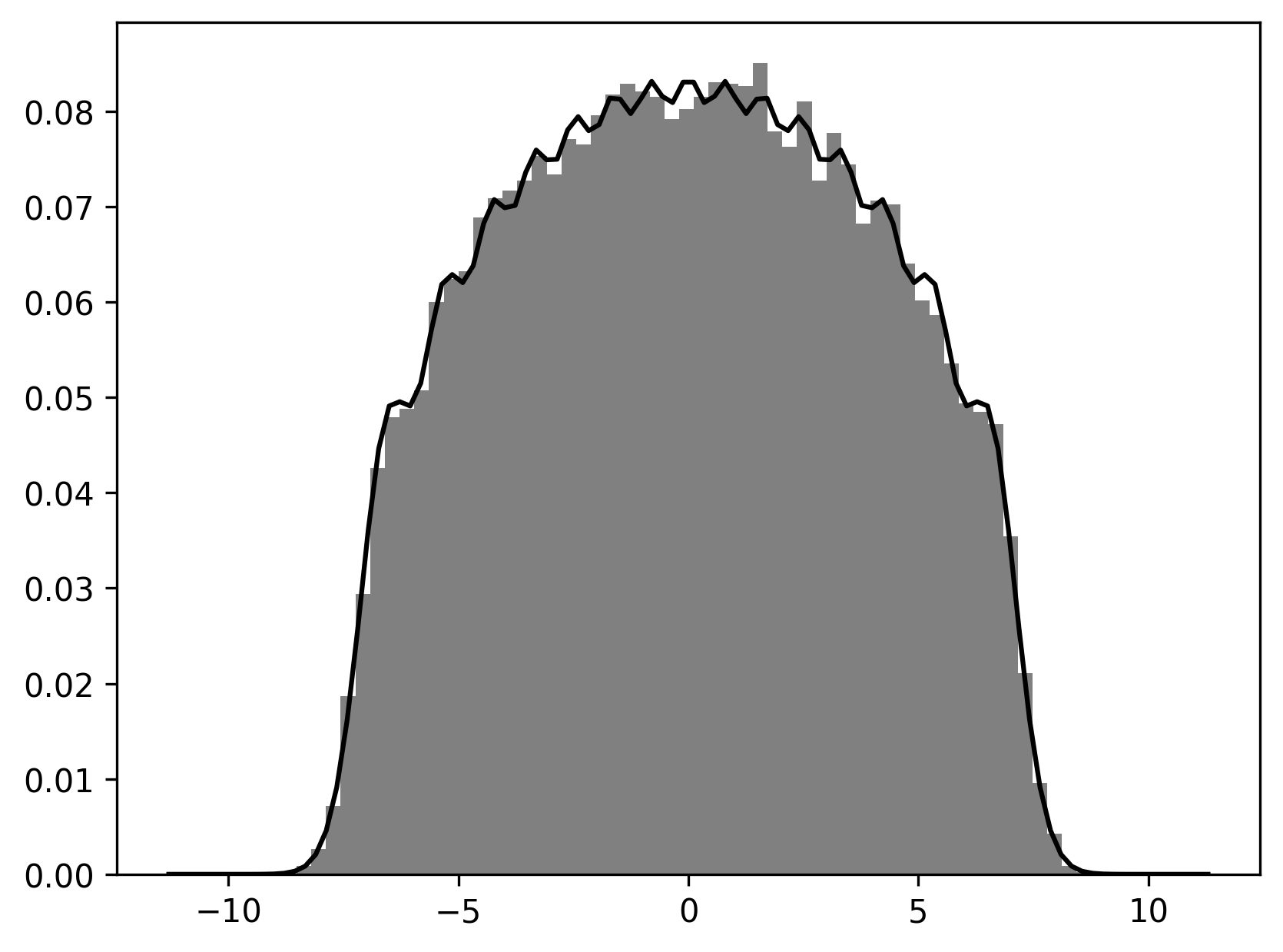}%
            \label{subfig:c}%
        }\hfill
        \subfloat[$n=20$]{%
            \includegraphics[width=.48\linewidth]{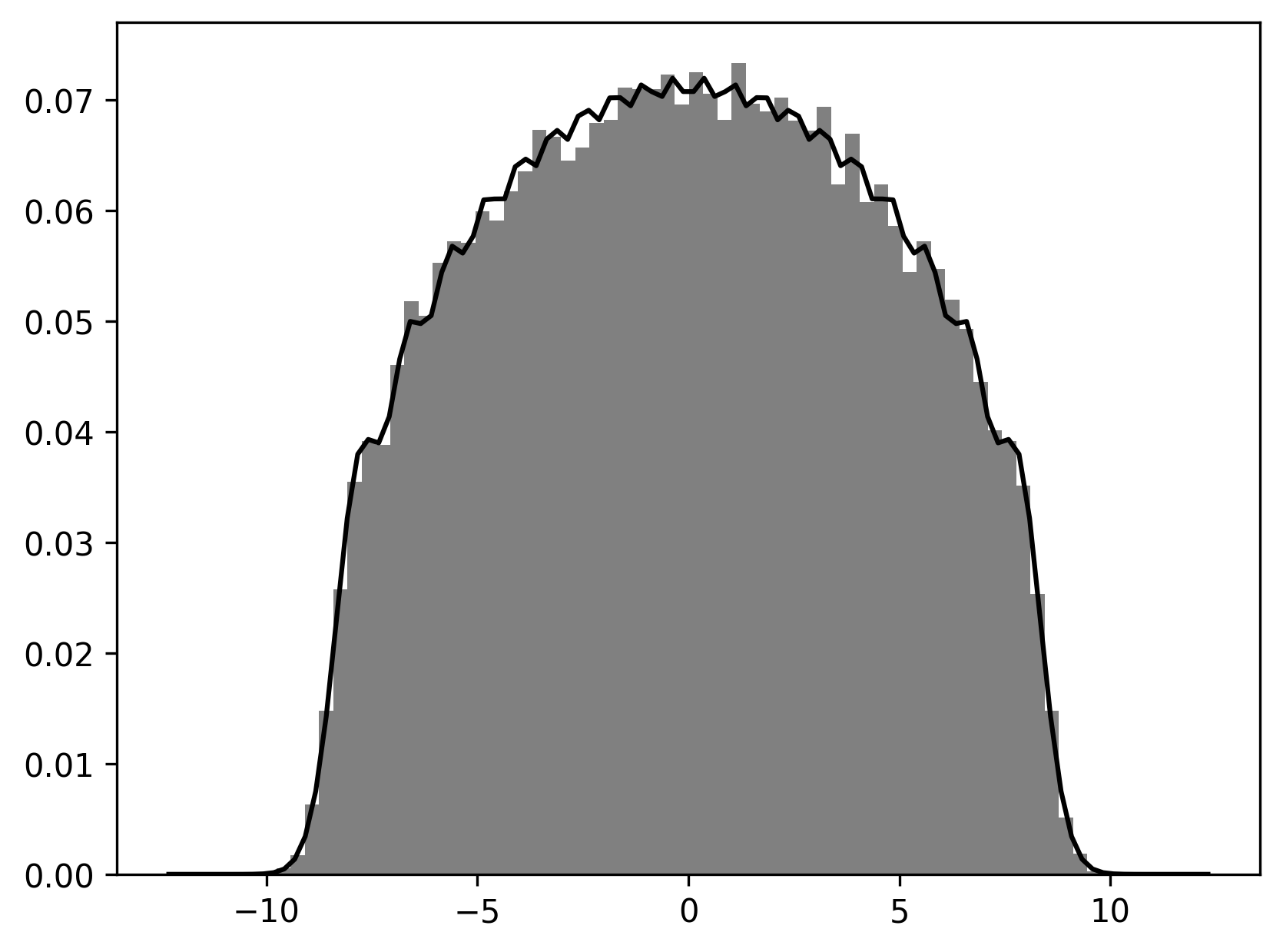}%
            \label{subfig:d}%
        }
        \caption{}
        \label{fig:fig}
    \end{figure*}

\bibliographystyle{abbrv}
\bibliography{biblio}

\end{document}